\documentclass[11pt]{article}
\usepackage{amsmath,amssymb,amsthm}
\hfuzz=\maxdimen
\newtheorem{thm}{Theorem}[section]

\newtheorem{lem}[thm]{Lemma}

\theoremstyle{definition}

\begin{document}
\date{}

\title{$d$-Degree Erd\H{o}s-Ko-Rado theorem for finite vector spaces\footnote{Supported by  National Natural Science Foundation of China (12171028, 12371326).}}

\author{
 {\small Yunjing  Shan,}  {\small Junling  Zhou}\\
{\small School of Mathematics and Statistics}\\ {\small Beijing Jiaotong University}\\
  {\small Beijing  100044, China}\\
 {\small jlzhou@bjtu.edu.cn}\\
}

\maketitle

\begin{abstract}
Let $V$ be an $n$-dimensional vector space over the finite field $\mathbb{F}_{q}$ and let $\left[V\atop k\right]_q$ denote the family of all $k$-dimensional subspaces of $V$. A family $\mathcal{F}\subseteq \left[V\atop k\right]_q$ is called intersecting if for all $F$, $F'\in\mathcal{F}$, we have ${\rm dim}$$(F\cap F')\geq 1$. Let $\delta_{d}(\mathcal{F})$ denote the minimum degree in $\mathcal{F}$ of all $d$-dimensional subspaces.  In this paper we show that $\delta_{d}(\mathcal{F})\leq \left[n-d-1\atop k-d-1\right]$ in any intersecting family $\mathcal{F}\subseteq \left[V\atop k\right]_q$, where $k>d\geq 2$ and $n\geq 2k+1$.
\end{abstract}

{\bf  Key words}\  \ \  Erd\H{o}s-Ko-Rado theorem  \ \  $d$-degree \ \ $q$-Kneser graph  \ \  vector space \ \

\section{Introduction}

Let $X$ be an $n$-element set and let $\tbinom{X}{k}$ denote the set of all $k$-element subsets of $X$. A family $\mathcal{F}\subseteq \tbinom{X}{k}$ is called {\it intersecting} if for all $F$, $F'\in\mathcal{F}$, we have $|F\cap F'|\geq 1$. For a family $\mathcal{F}\subseteq \tbinom{X}{k}$ and a fixed subset $S\subseteq X$, let $$d_{S}(\mathcal{F}):=|\{F\in\mathcal{F}:S\subseteq F\}|$$ denote the $S$-degree of $\mathcal{F}$. For a fixed positive integer $d\leq k$,
$$\delta_{d}(\mathcal{F})=\min\left\{d_{S}(\mathcal{F}):|S|=d, S\subseteq X\right\}$$
denote the {\it minimum $d$-degree} of $\mathcal{F}$.

One of the basic results in extremal set theory is the well-known Erd\H{o}s-Ko-Rado theorem. The Erd\H{o}s-Ko-Rado theorem states that if $n\geq 2k+1$ and $\mathcal{F}\subseteq \tbinom{X}{k}$ is an intersecting family, then $|\mathcal{F}|\leq\tbinom{n-1}{k-1}$ with equality holding if and only if $$\mathcal{F}=\left\{F\in\tbinom{X}{k}:y\in F\right\}$$ for some $y\in X.$ Huang and Zhao \cite{deg} established the following degree analogue of the Erd\H{o}s-Ko-Rado theorem.

\begin{thm}\label{1}{\rm(\cite[Theorem 1.1]{deg})}
Let $k, n$ be positive integers with $n\geq 2k+1$. Suppose $\mathcal{F}\subseteq \tbinom{X}{k}$ is an intersecting family. Then
\begin{equation*}
\delta_{1}(\mathcal{F})\leq \tbinom{n-2}{k-2}.
\end{equation*}
Moreover, equality holds if and only if $\mathcal{F}=\left\{F\in\tbinom{X}{k}: x\in F\right\}$ for some $x\in X$.
\end{thm}

%Bounds for an $r$-cover-free family without block size restriction have been studied by several researchers. Sperner's theorem in \cite{sper} states that for a $1$-cover-free family $\mathcal{F}$,
%\begin{equation*}
%|\mathcal{F}| \leq \tbinom{n}{\lfloor\frac{n}{2}\rfloor}
%\end{equation*}
%for all positive integers $n\geq2$. This bound is tight.
For $d\geq 2$, Kupavskii \cite{d1} confirmed this speculation that $\delta_{d}(\mathcal{F})\leq \tbinom{n-d-1}{k-d-1}$ for sufficiently large $n$. Huang and Zhang \cite{d2} used techniques from spectral graph theory to improve the range of $n$ significantly.
\begin{thm}\label{dd2}{\rm(\cite[Theorem 1.1]{d2})}
Let $n, k, d$ be positive integers with $k>d\geq 2$ and $n\geq 2k+2d-3$. Suppose $\mathcal{F}\subseteq \tbinom{X}{k}$ is an intersecting family. Then
\begin{equation*}
\delta_{d}(\mathcal{F})\leq \tbinom{n-d-1}{k-d-1}.
\end{equation*}
\end{thm}

The problems in extremal set theory have natural extensions to families of subspaces over a finite field. Throughout the paper we always let $V$ be an $n$-dimensional vector space over the finite field $\mathbb{F}_{q}$. Let $\left[V\atop k\right]_q$ denote the family of all $k$-dimensional subspaces of $V$. For $m\in \mathbb{R}, k\in \mathbb{Z}^{+}$, define the Gaussian binomial coefficient by
$$\left[m\atop k\right]_{q}:=\prod\limits_{i=0}^{k-1}\dfrac{q^{m-i}-1}{q^{k-i}-1}.$$
Obviously, the size of $\left[V\atop k\right]_q$ is $\left[n\atop k\right]_{q}$. If $k=0$, let $\left[m\atop k\right]_{q}:=1$; if $k\in \mathbb{Z}^{-}$, let $\left[m\atop k\right]_{q}:=0$. If $k$ and $q$ are fixed, then $\left[m\atop k\right]_{q}$ is a continuous function of $m$ which is positive and strictly increasing when $m\geq k$. If there is no ambiguity, the subscript $q$ can be omitted.

We denote $S\leq T$ if $S$ is a subspace of $T$. For any two subspaces $S$, $T\leq V$, let $S+T$ denote the linear span of $S\cup T$ and let $S\oplus T$ denote the direct sum of $S$ and $T$ with $S\cap T=\{\textup{\textbf{0}}\}$. Let $\mathcal{F}\subseteq \left[V\atop k\right]$ be a family of subspaces, we say that $\mathcal{F}$ is {\it intersecting} if for all $F$, $F'\in\mathcal{F}$, we have ${\rm dim}$$(F\cap F')\geq 1$. For a family $\mathcal{F}\subseteq \left[V\atop k\right]$ and a fixed subspace $S$, let
\begin{equation*}
d_{S}(\mathcal{F}):=|\{F\in\mathcal{F}:S\leq F\}|
\end{equation*}
denote the {\it $S$-degree} of $\mathcal{F}$. For a fixed positive integer $d\leq k$, $$\delta_{d}(\mathcal{F})=\min\left\{d_{S}(\mathcal{F}):{\rm dim}(S)=d, S\leq V\right\}$$
denote the {\it minimum $d$-degree} of $\mathcal{F}$.

%While many results about sets have been generalized to vector spaces, not so much is known about their $q$-analogs because adapting combinatorial techniques to vector spaces can be challenging.

Frankl and Tokushige \cite{vd} proved a vector space version of Theorem \ref{1} along the same line as Huang and Zhao's proof in \cite{deg}.

\begin{thm}\label{3}{\rm(\cite[Theorem 2]{vd})}
Let $k, n$ be positive integers with $n\geq 2k+1$. Suppose $\mathcal{F}\subseteq \left[V\atop k\right]$ is an intersecting family. Then
\begin{equation*}
\delta_{1}(\mathcal{F})\leq \left[n-2\atop k-2\right].
\end{equation*}
Moreover, equality holds if and only if $\mathcal{F}=\left\{F\in \left[V\atop k\right]: E\leq F\right\}$ for some $E\in \left[V\atop 1\right]$.
\end{thm}

In this paper, we will generalize Theorem~\ref{dd2} to vector spaces using spectral graph theory. In the process of our proof we adapt Huang and Zhang's  original proof of  Theorem~\ref{dd2} for sets to the $q$-analog version. The main result is as follows.
\begin{thm}\label{4}
Let $n, k, d$ be positive integers with $k>d\geq 2$ and $n\geq 2k+1$. Suppose $\mathcal{F}\subseteq \left[V\atop k\right]$ is an intersecting family. Then
\begin{equation*}
\delta_{d}(\mathcal{F})\leq \left[n-d-1\atop k-d-1\right].
\end{equation*}
\end{thm}
It is obvious that the upper bound given by Theorem \ref{4} is attained by the family of all $k$-dimensional subspaces of $V$ that contain a fixed 1-dimensional subspace.

\section{ Preliminaries}
In this section, we recall some important conclusions and give the proofs of several technical lemmas which are essential in the following sections.

Hsieh\;\cite{ekr2} established a vector space analog of Erd\H{o}s-Ko-Rado theorem which asserts that if the dimension of $V$ is large enough, then the unique intersecting family in $\left[V\atop k\right]$ with maximum size consists of all $k$-dimensional subspaces containing a fixed 1-dimensional subspace.
\begin{thm}\label{BBC}{\rm(\cite[Theorem 4.4]{ekr2})}
Let $n\geq 2k+1$. Suppose $\mathcal{H}\subseteq \left[V\atop k\right]$ is an intersecting family. Then
\begin{equation*}
\big|\mathcal{H}\big|\leq\left[n-1\atop k-1\right].
 \end{equation*}
Moreover, equality holds if and only if $\mathcal{H}=\left\{H\in\left[V\atop k\right]:E\leq H\right\}$ for some $E\in\left[V\atop 1\right]$.
\end{thm}

It is routine to enumerate the subspaces that intersect a given subspace trivially.
\begin{lem}\label{l21}{\rm(\cite[Propositions 2.2]{chen})} Let $Z$ be an $m$-dimensional subspace of the $n$-dimensional vector space $V$ over~$\mathbb{F}_{q}$. For a positive integer $l$ with $m+l\leq n$, the number of~$l$-dimensional subspaces~~$W$~of~~$V$~such that ${\rm dim}(Z\cap W)=0$ is $q^{lm}\left[n-m\atop l\right]$.
\end{lem}

Many scholars have studied $q$-identities obtained from classical binomial identities. A well-known relation involving binomial coefficients is binomial theorem. In parallel, we note two similar relations involving Gaussian binomial coefficients.

\begin{lem}\label{l10}{\rm(\cite[Theorem 3.3]{huang})} Let $z$ be an indeterminate and $m$ be a positive integer. Then
\begin{equation*}
\prod\limits_{i=0}^{m-1}(1-q^{i}z)=\sum\limits_{j=0}^{m}(-1)^{j}q^{\binom{j}{2}}\left[m\atop j\right]z^{j},
\end{equation*}

\begin{equation*}
\frac{1}{\prod\limits_{i=0}^{m-1}(1-q^{i}z)}=\sum\limits_{j=0}^{\infty}\left[m+j-1\atop j\right]z^{j}.
\end{equation*}
\end{lem}

Let $G$ be the $q$-Kneser graph whose vertex set is $\left[V\atop k\right]$ and two vertices $F, F'$ are adjacent if ${\rm dim}(F\cap F')=0$. Let $A$ be $q^{-k^{2}}$ times the adjacency matrix of $G$. Then it is well-known (e.g. \cite{ekr vector} or \cite{kneser}) that $A$ has eigenvalues $\lambda_{i}$ with multiplicity $m_{i}:=\left[n\atop i\right]-\left[n\atop i-1\right]$, where
\begin{equation}\label{e2}
\lambda_{i}=(-1)^{i}q^{\binom{i}{2}-ki}\left[n-k-i\atop k-i\right],
\end{equation}
$i=0, 1, 2, \ldots, k$. Let $U$ be the vector space of dimension $\left[n\atop k\right]$ over $\mathbb{R}$ with coordinates indexed by $k$-dimensional subspaces of $V$. Then $U$ has
an orthogonal decomposition
\begin{equation*}
U=U_{0}\oplus U_{1}\oplus \cdots \oplus U_{k},
\end{equation*}
where $U_{i}$ is the eigenspace corresponding to $\lambda_{i}$. Moreover, the eigenspace $U_{0}$ of the eigenvalue $\lambda_{0}=\left[n-k\atop k\right]$ is spanned by the unit vector $\frac{\vec{1}}{\sqrt{\left[n\atop k\right]}}$, where $\vec{1} \in\mathbb{R}^{\left[n\atop k\right]}$ denotes the all ones vector.

Let $\vec{h}\in \mathbb{R}^{\left[n\atop k\right]}$ be the indicating vector of the intersecting family $\mathcal{F}$. For $i=0, 1, 2, \ldots, k$, let $\vec{h}_{i}$ be the projection of $\vec{h}$ onto the subspace $U_{i}$. We have
\begin{equation}\label{e1}
\vec{h}=\vec{h}_{0}+ \vec{h}_{1}+ \cdots +\vec{h}_{k},
\end{equation}
\begin{equation*}
A\vec{h}_{i}=(-1)^{i}q^{\binom{i}{2}-ki}\left[n-k-i\atop k-i\right]\vec{h}_{i}.
\end{equation*}
Endow $V$ with the Euclidean inner product. It is obvious that $\langle \vec{h}_{i},\vec{h}_{j}\rangle=\vec{h}_{i}^{T}\vec{h}_{j}=0$ if $i\neq j$.

For $0\leq i\leq j\leq k$, let us define the matrices $W_{i,j}$ of size $\left[n\atop i\right]\times \left[n\atop j\right]$ with rows indexed by the $i$-dimensional subspaces of $V$, columns indexed by the $j$-dimensional subspaces of $V$, and
\begin{equation*}
(W_{i,j})_{S,T}=
\begin{cases}
1, &\text{if~$S\leq T$},\\[.3cm]
0, &\text{otherwise}.
\end{cases}\
\end{equation*}
Similarly, we also define the matrices $\overline{W}_{i,j}$ of size $\left[n\atop i\right]\times \left[n\atop j\right]$ by
\begin{equation*}
(\overline{W}_{i,j})_{S,T}=
\begin{cases}
1, &\text{if~$S\cap T=\{\textbf{0}\}$},\\[.3cm]
0, &\text{otherwise}.
\end{cases}\
\end{equation*}
Let us recall some identities involving the two matrices defined above.
\begin{lem}\label{l3}{\rm(\cite{ekr vector})}
For $0\leq i\leq j\leq r\leq k$, the following hold.
\begin{itemize}
\item[\rm(i)] $\overline{W}_{i,j}=\sum\limits_{m=0}^{i}(-1)^{m}q^{\binom{m}{2}}W_{m,i}^{T}W_{m,j}$.
\item[\rm(ii)] $W_{i,j}=\sum\limits_{m=0}^{i}(-1)^{m}q^{\binom{m+1}{2}-mi}W_{m,i}^{T}\overline{W}_{m,j}$.
\item[\rm(iii)]
$W_{i,j}W_{j,r}=\left[r-i\atop j-i\right]W_{i,r}$.
\item[\rm(iv)]
${\rm rowsp}(W_{i,k})=U_{0}\oplus U_{1}\oplus \cdots \oplus U_{i}$, where ${\rm rowsp}(W_{i,k})$ denotes the row space of $W_{i,k}$ over $\mathbb{R}$.
\end{itemize}
\end{lem}

\begin{lem}\label{l6}{\rm(\cite{lym})} Let $0\leq j\leq i\leq k<n$. Then
\begin{equation}\label{e7}
W_{i,k}W_{j,k}^{T}=\sum\limits_{m=0}^{j}q^{m(k+m-i-j)}\left[n-i-j\atop n-k-m\right]W_{m,i}^{T}W_{m,j}.
\end{equation}
\end{lem}

In the following Lemma, we prove that $U_{j}$ is an eigenspace for the matrix $W_{i,k}^{T}W_{i,k}$ for its eigenvalue $q^{j(k-i)}\left[k-j\atop k-i\right]\left[n-i-j\atop k-i\right]$ if $j\leq i$. This
eigenvalue vanishes if $i<j$.
\begin{lem}\label{l4}
For $0\leq i, j\leq k$, $\vec{v}\in U_{j}$, we have
\begin{equation}\label{e9}
W_{i,k}^{T}W_{i,k}\vec{v}=
\begin{cases}
q^{j(k-i)}\left[k-j\atop k-i\right]\left[n-i-j\atop k-i\right]\vec{v}, &\text{if~~$0\leq j\leq i$},\\[.3cm]
\vec{0}, &\text{if~~$i<j$}.
\end{cases}\
\end{equation}
\end{lem}
\begin{proof}~Suppose that $i<j$. We have
\begin{equation*}
{\rm rowsp}(W_{j-1,k})=U_{0}\oplus U_{1}\oplus \cdots \oplus U_{j-1}
\end{equation*}
by lemma~\ref{l3}(iv). Since $\vec{v}\in U_{j}$,
\begin{equation*}
W_{j-1,k}\vec{v}=\vec{0}.
\end{equation*}
 Hence,
\begin{equation*}
W_{i,k}^{T}W_{i,k}\vec{v}=W_{i,k}^{T}\cdot\frac{W_{i,j-1}W_{j-1,k}}{\left[k-i\atop j-1-i\right]}\cdot\vec{v}=\vec{0}.
\end{equation*}
by lemma~\ref{l3}(iii).

Suppose that $0\leq j\leq i$. By lemma~\ref{l3}(iv) and $\vec{v}\in U_{j}$. We have
\begin{equation}\label{e6}
\vec{v}=W_{j,k}^{T}\vec{u}
\end{equation}
for some $\vec{u} \in \mathbb{R}^{\left[n\atop j\right]}$. We claim that $W_{m,j}\vec{u}=\vec{0}$ for $m<j$. It is easily checked that $W_{m,k}\vec{v}=\vec{0}$ for any $m<j$ by lemma~\ref{l3}(iv). Then
\begin{equation}\label{e5}
\overline{W}_{m,k}\vec{v}=\sum\limits_{r=0}^{m}(-1)^{r}q^{\binom{r}{2}}W_{r,m}^{T}W_{r,k}\vec{v}=\vec{0}
\end{equation}
by lemma~\ref{l3}(i). For any $m<j\leq k$, $S\in\left[V\atop m\right]$ and $T\in\left[V\atop j\right]$,
\begin{equation}\label{e100}
\begin{array}{rl}
\left(\overline{W}_{m,k}W_{j,k}^{T}\right)_{S, T}\!\!\!\!&=\sum\limits_{F\in\left[V\atop k\right]}\left(\overline{W}_{m,k}\right)_{S, F}\left(W_{j,k}^{T}\right)_{F, T}\\[.3cm]
&=\left|\left\{F\in\left[V\atop k\right]:F\cap S=\{\textbf{0}\}, T\leq F\right\}\right|\\[.3cm]
&=\left|\left\{P\in\left[W\atop k-j\right]:P\cap S=\{\textbf{0}\}\right\}\right|\left(\overline{W}_{m,j}\right)_{S, T}\\[.3cm]
&=q^{m(k-j)}\left[n-m-j\atop k-j\right]\left(\overline{W}_{m,j}\right)_{S, T}
\end{array}
\end{equation}
by lemma~\ref{l21}, where $T\oplus W=V$, ${\rm dim}(W)=n-j$ and $S\leq W$ if $\left(\overline{W}_{m,j}\right)_{S, T}=1$. Then,
\begin{equation*}
\vec{0}\overset{\eqref{e5}}{=}\overline{W}_{m,k}\vec{v}\overset{\eqref{e6}}{=}\overline{W}_{m,k}W_{j,k}^{T}\vec{u}\overset{\eqref{e100}}{=}q^{m(k-j)}\left[n-m-j\atop k-j\right]\overline{W}_{m,j}\vec{u},
\end{equation*}
that is $\overline{W}_{m,j}\vec{u}=\vec{0}$ for any $m<j$. Hence,
\begin{equation}\label{e8}
W_{m,j}\vec{u}=\sum\limits_{r=0}^{m}(-1)^{r}q^{\binom{r+1}{2}-rm}W_{r,m}^{T}\overline{W}_{r,j}\vec{u}=\vec{0}
\end{equation}
by lemma~\ref{l3}(ii), that is, the claim holds. Hence, for $0\leq j\leq i\leq k$, applying Lemma~\ref{l3}(iii) yields
\begin{equation*}
\begin{array}{rl}
W_{i,k}^{T}W_{i,k}\vec{v}\!\!\!\!&\overset{\eqref{e6}}{=}W_{i,k}^{T}W_{i,k}W_{j,k}^{T}\vec{u}\\[.3cm]
&\overset{\eqref{e7}}{=}W_{i,k}^{T}\left(\sum\limits_{r=0}^{j}q^{r(k+r-i-j)}\left[n-i-j\atop n-k-r\right]W_{r,i}^{T}W_{r,j}\right)\vec{u}\\[.3cm]
&\overset{\eqref{e8}}{=}q^{j(k-i)}\left[n-i-j\atop n-k-j\right]W_{i,k}^{T}W_{j,i}^{T}W_{j,j}\vec{u}\\[.3cm]
&=q^{j(k-i)}\left[n-i-j\atop k-i\right]\left(W_{j,i}W_{i,k}\right)^{T}\vec{u}\\[.3cm]
&=q^{j(k-i)}\left[k-j\atop i-j\right]\left[n-i-j\atop k-i\right]W_{j,k}^{T}\vec{u}\\[.3cm]
&\overset{\eqref{e6}}{=}q^{j(k-i)}\left[k-j\atop k-i\right]\left[n-i-j\atop k-i\right]\vec{v}.
\end{array}
\end{equation*}
Note that $W_{j,j}$ is the identity matrix.
\end{proof}

\begin{lem}\label{l8}
Let $k>d$ and $n\geq 2k$. We have
\begin{equation}\label{e10}
\left(\!W_{d,k}\vec{h}\right)^{T}\!\overline{W}_{d,d}\!\left(\!W_{d,k}\vec{h}\right)\!=\!\sum\limits_{r=0}^{d}(\!-1)^{r}\!q^{kr+(d-r)^{2}\!-\!\binom{r+1}{2}}\!\left[k\!-\!r\atop d\!-\!r\right]\!\!\!\left[n\!-\!d\!-\!r\atop d\!-\!r\right]\!\!\!\left[n\!-\!d\!-\!r\atop k\!-\!d\right]\!\!\|\vec{h}_{r}\|^{2}.
\end{equation}
\end{lem}

\begin{proof}~
By Lemma~\ref{l3}(i),(iii), we have

\begin{equation}\label{e101}
\begin{array}{rl}
W_{d,k}^{T}\overline{W}_{d,d}W_{d,k}
\!\!\!\!&=W_{d,k}^{T}\left(\sum\limits_{i=0}^{d}(-1)^{i}q^{\binom{i}{2}}W_{i,d}^{T}W_{i,d}\right)W_{d,k}\\[.3cm]
&=\sum\limits_{i=0}^{d}(-1)^{i}q^{\binom{i}{2}}\left(W_{i,d}W_{d,k}\right)^{T}W_{i,d}W_{d,k}\\[.3cm]
&=\sum\limits_{i=0}^{d}(-1)^{i}q^{\binom{i}{2}}\left[k-i\atop d-i\right]^{2}W_{i,k}^{T}W_{i,k}.
\end{array}
\end{equation}
Then
\begin{equation*}
\begin{array}{rl}
&~~~\left(W_{d,k}\vec{h}\right)^{T}\overline{W}_{d,d}\left(W_{d,k}\vec{h}\right)\\[.3cm]
&=\vec{h}^{T}\left(W_{d,k}^{T}\overline{W}_{d,d}W_{d,k}\right)\vec{h}\\[.3cm]
&\overset{\eqref{e101}}{=}\vec{h}^{T}\left(\sum\limits_{i=0}^{d}(-1)^{i}q^{\binom{i}{2}}\left[k-i\atop d-i\right]^{2}W_{i,k}^{T}W_{i,k}\right)\vec{h}\\[.3cm]
&\overset{\eqref{e1}}{=}\left(\sum\limits_{j=0}^{k}\vec{h}_{j}^{T}\right)\left(\sum\limits_{i=0}^{d}(-1)^{i}q^{\binom{i}{2}}\left[k-i\atop d-i\right]^{2}W_{i,k}^{T}W_{i,k}\right)\left(\sum\limits_{r=0}^{k}\vec{h}_{r}\right)\\[.3cm]
&=\left(\sum\limits_{j=0}^{k}\vec{h}_{j}^{T}\right)\left(\sum\limits_{i=0}^{d}\sum\limits_{r=0}^{k}(-1)^{i}q^{\binom{i}{2}}\left[k-i\atop d-i\right]^{2}W_{i,k}^{T}W_{i,k}\vec{h}_{r}\right)\\[.3cm]
&\overset{\eqref{e9}}{=}\left(\sum\limits_{j=0}^{k}\vec{h}_{j}^{T}\right)\left(\sum\limits_{i=0}^{d}\sum\limits_{r=0}^{i}(-1)^{i}q^{\binom{i}{2}}\left[k-i\atop d-i\right]^{2}\cdot q^{r(k-i)}\left[k-r\atop k-i\right]\left[n-i-r\atop k-i\right]\vec{h}_{r}\right)\\[.3cm]
&=\sum\limits_{i=0}^{d}\sum\limits_{r=0}^{i}(-1)^{i}q^{\binom{i}{2}}q^{r(k-i)}\left[k-i\atop d-i\right]^{2}\left[k-r\atop k-i\right]\left[n-i-r\atop k-i\right]\|\vec{h}_{r}\|^{2}\\[.3cm]
&=\sum\limits_{i=0}^{d}(-1)^{i}q^{\binom{i}{2}}\left[k-i\atop d-i\right]^{2}\left(\sum\limits_{r=0}^{i}q^{r(k-i)}\left[k-r\atop k-i\right]\left[n-i-r\atop k-i\right]\|\vec{h}_{r}\|^{2}\right)\\[.3cm]
&=\sum\limits_{r=0}^{d}\|\vec{h}_{r}\|^{2}\left(\sum\limits_{i=r}^{d}(-1)^{i}q^{\binom{i}{2}+r(k-i)}\left[k-i\atop d-i\right]^{2}\left[k-r\atop k-i\right]\left[n-i-r\atop k-i\right]\right).
\end{array}
\end{equation*}

We only need to prove that
\begin{equation*}
\begin{array}{rl}
\sum\limits_{i=r}^{d}(-1)^{i}q^{\binom{i}{2}+r(k-i)}\left[k-i\atop d-i\right]^{2}\left[k-r\atop k-i\right]\left[n-i-r\atop k-i\right]=(-1)^{r}q^{kr+(d-r)^{2}-\binom{r+1}{2}}\left[k-r\atop d-r\right]\left[n-d-r\atop d-r\right]\left[n-d-r\atop k-d\right].
 \end{array}
\end{equation*}
Since
\begin{equation*}
\begin{array}{rl}
&~~~~\frac{\left[k-i\atop d-i\right]^{2}\left[k-r\atop k-i\right]\left[n-i-r\atop k-i\right]}{\left[k-r\atop d-r\right]\left[n-d-r\atop k-d\right]}\\[.3cm]
&=\frac{\prod\limits_{j=i}^{d-1}\left(q^{k-j}-1\right)^{2}\cdot\prod\limits_{j=r}^{r+k-i-1}\left(q^{k-j}-1\right)\cdot\prod\limits_{j=i+r}^{r+k-1}\left(q^{n-j}-1\right)\cdot\prod\limits_{j=r}^{d-1}\left(q^{d-j}-1\right)\cdot\prod\limits_{j=d}^{k-1}\left(q^{k-j}-1\right)}{\prod\limits_{j=i}^{d-1}\left(q^{d-j}-1\right)^{2}\cdot\prod\limits_{j=i}^{k-1}\left(q^{k-j}-1\right)^{2}\cdot\prod\limits_{j=r}^{d-1}\left(q^{k-j}-1\right)\cdot\prod\limits_{j=d+r}^{r+k-1}\left(q^{n-j}-1\right)}\\[.3cm]
&=\frac{\prod\limits_{j=i}^{d-1}\left(q^{k-j}-1\right)\cdot\prod\limits_{j=r}^{r+k-i-1}\left(q^{k-j}-1\right)\cdot\prod\limits_{j=i+r}^{r+d-1}\left(q^{n-j}-1\right)\cdot\prod\limits_{j=r}^{i-1}\left(q^{d-j}-1\right)}{\prod\limits_{j=i}^{d-1}\left(q^{d-j}-1\right)\cdot\prod\limits_{j=i}^{k-1}\left(q^{k-j}-1\right)\cdot\prod\limits_{j=r}^{d-1}\left(q^{k-j}-1\right)}\\[.3cm]
&=\frac{\prod\limits_{j=r}^{r+k-i-1}\left(q^{k-j}-1\right)\cdot\prod\limits_{j=i+r}^{r+d-1}\left(q^{n-j}-1\right)\cdot\prod\limits_{j=r}^{i-1}\left(q^{d-j}-1\right)}{\prod\limits_{j=i}^{d-1}\left(q^{d-j}-1\right)\cdot\prod\limits_{j=r}^{k-1}\left(q^{k-j}-1\right)}\\[.3cm]
&=\frac{\prod\limits_{j=i+r}^{r+d-1}\left(q^{n-j}-1\right)\cdot\prod\limits_{j=r}^{i-1}\left(q^{d-j}-1\right)}{\prod\limits_{j=i}^{d-1}\left(q^{d-j}-1\right)\cdot\prod\limits_{j=r+k-i}^{k-1}\left(q^{k-j}-1\right)}\\[.3cm]
&=\left[n-i-r\atop d-i\right]\left[d-r\atop i-r\right],
 \end{array}
\end{equation*}
the identity we would like to prove is equivalent to
\begin{equation}\label{e11}
\begin{array}{rl}
\sum\limits_{i=r}^{d}(-1)^{i-r}q^{\binom{i-r}{2}}\left[d-r\atop i-r\right]\left[n-i-r\atop d-i\right]=q^{(d-r)^{2}}\left[n-d-r\atop d-r\right].
 \end{array}
\end{equation}
Letting $z$ be an indeterminate, for $r\leq d< k$, obviously,
\begin{equation*}
\begin{array}{rl}
\frac{\prod\limits_{s=0}^{d-r-1}(1-q^{s}z)}{\prod\limits_{s=0}^{n-d-r}(1-q^{s}z)}=\frac{1}{\prod\limits_{s=d-r}^{n-d-r}(1-q^{s}z)},
\end{array}
\end{equation*}
which is equivalent to
\begin{equation*}
\begin{array}{rl}
\left(\sum\limits_{s=0}^{d-r}(-1)^{s}q^{\binom{s}{2}}\left[d-r\atop s\right]z^{s}\right)\left(\sum\limits_{t=0}^{\infty}\left[n-d-r+t\atop t\right]z^{t}\right)=\sum\limits_{m=0}^{\infty}\left[n-2d+m\atop m\right](q^{d-r}z)^{m}
 \end{array}
\end{equation*}
by Lemma~\ref{l10}. Then \eqref{e11} follows from comparing the coefficients of $z^{i-r}\cdot z^{d-i}=z^{d-r}~(r\leq i\leq d)$ in the above equality. This completes the proof.
\end{proof}

\section{ Proof of Theorem~\ref{4}}
In this section, we will prove Theorem\;\ref{4}. We first establish two lemmas that contain two inequalities involving $\|\vec{h}_{i}\|.$ The first inequality for intersecting families which follows from a Hoffman bound type argument \cite{H1}. The second inequality is derived directly from the assumption that $\delta_{d}(\mathcal{F})> \left[n-d-1\atop k-d-1\right]$.
\begin{lem}\label{l1}
Let $k>d\geq1, n\geq 2k+1$. Suppose that $\mathcal{F}\subseteq \left[V\atop k\right]$ is an intersecting family. Then
\begin{equation*}
0>-\frac{q^{k}-q^{d}}{q^{n}-q^{k}}\left[n-k\atop k\right]|\mathcal{F}|+\sum_{i=0}^{d}\left(\frac{q^{k}-q^{d}}{q^{n}-q^{k}}\left[n-k\atop k\right]+(-1)^{i}q^{\binom{i}{2}-ki}\left[n-k-i\atop k-i\right]\right)\|\vec{h}_{i}\|^{2}.
\end{equation*}
\end{lem}

\begin{proof}~
Since $\mathcal{F}$ is an intersecting family and $\vec{h}$ is the indicating vector of $\mathcal{F}$,
\begin{equation*}
|\mathcal{F}|=\langle \vec{h}, \vec{h}\rangle\overset{\eqref{e1}}{=}\left(\sum\limits_{i=0}^{k}\vec{h}_{i}^{T}\right)\left(\sum\limits_{j=0}^{k}\vec{h}_{j}\right)=\sum\limits_{i=0}^{k}\|\vec{h}_{i}\|^{2}.
\end{equation*}
Since $A$ is $q^{-k^{2}}$ times the adjacency matrix of the $q$-Kneser graph, we have $\vec{h}^{T}A\vec{h}=0$. Then
\begin{equation*}
\begin{array}{rl}
0\!\!\!\!&=\vec{h}^{T}A\vec{h}\\[.3cm]
&\overset{\eqref{e1}}{=}\left(\sum\limits_{i=0}^{k}\vec{h}_{i}^{T}\right)\left(\sum\limits_{j=0}^{k}A\vec{h}_{j}\right)\\[.3cm]
&=\left(\sum\limits_{i=0}^{k}\vec{h}_{i}^{T}\right)\left(\sum\limits_{j=0}^{k}\lambda_{j}\vec{h}_{j}\right)\\[.3cm]
&=\sum\limits_{i=0}^{k}\lambda_{i}\|\vec{h}_{i}\|^{2}\\[.3cm]
&\overset{\eqref{e2}}{=}\sum\limits_{i=0}^{k}(-1)^{i}q^{\binom{i}{2}-ki}\left[n-k-i\atop k-i\right]\|\vec{h}_{i}\|^{2}\\[.3cm]
&=\sum\limits_{i=0}^{d}(-1)^{i}q^{\binom{i}{2}-ki}\left[n-k-i\atop k-i\right]\|\vec{h}_{i}\|^{2}+\sum\limits_{i=d+1}^{k}(-1)^{i}q^{\binom{i}{2}-ki}\left[n-k-i\atop k-i\right]\|\vec{h}_{i}\|^{2}\\[.3cm]
&\geq\sum\limits_{i=0}^{d}(-1)^{i}q^{\binom{i}{2}-ki}\left[n-k-i\atop k-i\right]\|\vec{h}_{i}\|^{2}-q^{\binom{d+1}{2}-k(d+1)}\left[n-k-d-1\atop k-d-1\right]\sum\limits_{i=d+1}^{k}\|\vec{h}_{i}\|^{2}\\[.3cm]
&>\sum\limits_{i=0}^{d}(-1)^{i}q^{\binom{i}{2}-ki}\left[n-k-i\atop k-i\right]\|\vec{h}_{i}\|^{2}-\frac{q^{k}-q^{d}}{q^{n}-q^{k}}\left[n-k\atop k\right]\sum\limits_{i=d+1}^{k}\|\vec{h}_{i}\|^{2}\\[.3cm]
&=\sum\limits_{i=0}^{d}(-1)^{i}q^{\binom{i}{2}-ki}\left[n-k-i\atop k-i\right]\|\vec{h}_{i}\|^{2}-\frac{q^{k}-q^{d}}{q^{n}-q^{k}}\left[n-k\atop k\right]\left(|\mathcal{F}|-\sum\limits_{i=0}^{d}\|\vec{h}_{i}\|^{2}\right)\\[.3cm]
&=-\frac{q^{k}-q^{d}}{q^{n}-q^{k}}\left[n-k\atop k\right]|\mathcal{F}|+\sum\limits_{i=0}^{d}\left(\frac{q^{k}-q^{d}}{q^{n}-q^{k}}\left[n-k\atop k\right]+(-1)^{i}q^{\binom{i}{2}-ki}\left[n-k-i\atop k-i\right]\right)\|\vec{h}_{i}\|^{2}
\end{array}
\end{equation*}
by Lemma \ref{l200} in the Appendix.
\end{proof}

\begin{lem}\label{l2}
Suppose that $\mathcal{F}\subseteq \left[V\atop k\right]$ with $\delta_{d}(\mathcal{F})> \left[n-d-1\atop k-d-1\right]$. Then
\begin{equation*}
\begin{array}{rl}
0<\!\!\!\!&\sum\limits_{i=0}^{d}(-1)^{i}q^{ki+(d-i)^{2}-\binom{i+1}{2}}\left[k-i\atop d-i\right]\left[n-d-i\atop d-i\right]\left[n-d-i\atop k-d\right]\|\vec{h}_{r}\|^{2}-2q^{d^{2}}\left[k\atop d\right]\left[n-d\atop d\right]\left[n-d-1\atop k-d-1\right]\left|\mathcal{F}\right|\\[.3cm]
&+q^{d^{2}}\left[n\atop d\right]\left[n-d\atop d\right]\left[n-d-1\atop k-d-1\right]^{2}.
\end{array}
\end{equation*}
\end{lem}

\begin{proof}~We calculate
\begin{equation*}
\left|\left\{(S,F):S\in \left[V\atop d\right], F\in \mathcal{F}, S\leq F\right\}\right|
\end{equation*}
in two ways and get the following equation
\begin{equation}\label{e3}
\sum\limits_{S\in \left[V\atop d\right]}d_{S}(\mathcal{F})=\left[k\atop d\right]\left|\mathcal{F}\right|.
\end{equation}

We claim that
\begin{equation}\label{e106}
\begin{array}{rl}
\sum\limits_{S,T\in \left[V\atop d\right] \atop S\cap T=\{\textbf{0}\}}d_{S}(\mathcal{F})d_{T}=\left(W_{d,k}\vec{h}\right)^{T}\overline{W}_{d,d}\left(W_{d,k}\vec{h}\right).
\end{array}
\end{equation}
Since
\begin{equation*}
\begin{array}{rl}
\left(W_{d,k}\vec{h}\right)^{T}\overline{W}_{d,d}\left(W_{d,k}\vec{h}\right)
&=\sum\limits_{S,T\in \left[V\atop d\right]}\left(W_{d,k}\vec{h}\right)_{S}^{T}\left(\overline{W}_{d,d}\right)_{S,T}\left(W_{d,k}\vec{h}\right)_{T}\\[.3cm]
&=\sum\limits_{S,T\in \left[V\atop d\right]}d_{S}(\mathcal{F})\left(\overline{W}_{d,d}\right)_{S,T} d_{T}\\[.3cm]
&=\sum\limits_{S,T\in \left[V\atop d\right] \atop S\cap T=\{\textbf{0}\}}d_{S}(\mathcal{F})d_{T},
\end{array}
\end{equation*}
the claim holds. Since $\delta_{d}(\mathcal{F})> \left[n-d-1\atop k-d-1\right]$, we have
\begin{equation}\label{e4}
d_{S}(\mathcal{F})> \left[n-d-1\atop k-d-1\right]
\end{equation}
for any $S\in \left[V\atop d\right]$. Then

\begin{equation*}
\begin{array}{rl}
0\!\!\!\!&\overset{\eqref{e4}}{<}\sum\limits_{S,T\in \left[V\atop d\right] \atop S\cap T=\{\textbf{0}\}}\left(d_{S}(\mathcal{F})-\left[n-d-1\atop k-d-1\right]\right)\left(d_{T}-\left[n-d-1\atop k-d-1\right]\right)\\[.3cm]
&=\sum\limits_{S,T\in \left[V\atop d\right] \atop S\cap T=\{\textbf{0}\}}\left(d_{S}(\mathcal{F})d_{T}-\left(d_{S}(\mathcal{F})+d_{T}\right)\left[n-d-1\atop k-d-1\right]\right)\\
& \qquad +\left|\left\{(S,T):S,T\in \left[V\atop d\right], S\cap T=\{\textbf{0}\}\right\}\right|\cdot\left[n-d-1\atop k-d-1\right]^{2}\\[.3cm]
&=\sum\limits_{S,T\in \left[V\atop d\right] \atop S\cap T=\{\textbf{0}\}}d_{S}(\mathcal{F})d_{T}-\sum\limits_{S,T\in \left[V\atop d\right] \atop S\cap T=\{\textbf{0}\}}\left(d_{S}(\mathcal{F})+d_{T}\right)\left[n-d-1\atop k-d-1\right]+q^{d^{2}}\left[n\atop d\right]\left[n-d\atop d\right]\left[n-d-1\atop k-d-1\right]^{2}\\[.3cm]
&=\sum\limits_{S,T\in \left[V\atop d\right] \atop S\cap T=\{\textbf{0}\}}d_{S}(\mathcal{F})d_{T}-2\left[n-d-1\atop k-d-1\right]\sum\limits_{S\in \left[V\atop d\right]}\left(d_{S}(\mathcal{F})\cdot\left|\left\{T\in \left[V\atop d\right]:S\cap T=\{\textbf{0}\}\right\}\right|\right)\\
& \qquad +q^{d^{2}}\left[n\atop d\right]\left[n-d\atop d\right]\left[n-d-1\atop k-d-1\right]^{2}\\[.3cm]
&\overset{\eqref{e106}}{=}\left(W_{d,k}\vec{h}\right)^{T}\overline{W}_{d,d}\left(W_{d,k}\vec{h}\right)-2q^{d^{2}}\left[n-d\atop d\right]\left[n-d-1\atop k-d-1\right]\sum\limits_{S\in \left[V\atop d\right]}d_{S}(\mathcal{F})\\
& \qquad +q^{d^{2}}\left[n\atop d\right]\left[n-d\atop d\right]\left[n-d-1\atop k-d-1\right]^{2}\\[.3cm]
&\overset{\eqref{e3}}{=}\left(W_{d,k}\vec{h}\right)^{T}\overline{W}_{d,d}\left(W_{d,k}\vec{h}\right)-2q^{d^{2}}\left[k\atop d\right]\left[n-d\atop d\right]\left[n-d-1\atop k-d-1\right]\left|\mathcal{F}\right|+q^{d^{2}}\left[n\atop d\right]\left[n-d\atop d\right]\left[n-d-1\atop k-d-1\right]^{2}
\end{array}
\end{equation*}
by Lemma~\ref{l21}. Then the result holds by \eqref{e10}.
\end{proof}

\noindent\emph{{\textbf{Proof of Theorem~\ref{4}.}}}~
Suppose that $\mathcal{F}\subseteq \left[V\atop k\right]$ is an intersecting family with $\delta_{d}(\mathcal{F})> \left[n-d-1\atop k-d-1\right]$. Let
\begin{equation}\label{a}
a_{i}=\frac{q^{k}-q^{d}}{q^{n}-q^{k}}\left[n-k\atop k\right]+(-1)^{i}q^{\binom{i}{2}-ki}\left[n-k-i\atop k-i\right],
\end{equation}
\begin{equation}\label{b}
b_{i}=(-1)^{i}q^{ki+(d-i)^{2}-\binom{i+1}{2}}\left[k-i\atop d-i\right]\left[n-d-i\atop d-i\right]\left[n-d-i\atop k-d\right],
\end{equation}
\begin{equation}\label{cf}
c=\frac{q^{k}-q^{d}}{q^{n}-q^{k}}\left[n-k\atop k\right], f=2q^{d^{2}}\left[k\atop d\right]\left[n-d\atop d\right]\left[n-d-1\atop k-d-1\right],
\end{equation}
\begin{equation*}
g=q^{d^{2}}\left[n\atop d\right]\left[n-d\atop d\right]\left[n-d-1\atop k-d-1\right]^{2},
\end{equation*}
we have
\begin{equation}\label{e12}
0>\sum\limits_{i=0}^{d}a_{i}\|\vec{h_{i}}\|^{2}-c|\mathcal{F}|,
\end{equation}
\begin{equation}\label{e13}
0<\sum\limits_{i=0}^{d}b_{i}\|\vec{h_{i}}\|^{2}-f|\mathcal{F}|+g
\end{equation}
by Lemmas~\ref{l1}, \ref{l2}. It is easily checked that $b_{1}, a_{1}=\frac{1-q^{d}}{q^{n}-q^{k}}\left[n-k\atop k\right]$ are negative. We subtract \eqref{e12} multiplied by $\frac{b_{1}}{a_{1}}$ from \eqref{e13} and then
\begin{equation*}
0<\left(b_{0}-\frac{a_{0}b_{1}}{a_{1}}\right)\|\vec{h_{0}}\|^{2}+\sum\limits_{i=2}^{d}\left(b_{i}-\frac{a_{i}b_{1}}{a_{1}}\right)\|\vec{h_{i}}\|^{2}-\left(f-\frac{cb_{1}}{a_{1}}\right)|\mathcal{F}|+g.
\end{equation*}
We have
\begin{equation}\label{e14}
0<\left(b_{0}-\frac{a_{0}b_{1}}{a_{1}}\right)\|\vec{h_{0}}\|^{2}-\left(f-\frac{cb_{1}}{a_{1}}\right)|\mathcal{F}|+g
\end{equation}
by Lemma \ref{c4} in the Appendix. Since
\begin{equation*}
\begin{array}{rl}
\|\vec{h_{0}}\|^{2}=\langle\vec{h_{0}}, \frac{\vec{1}}{\sqrt{\left[n\atop k\right]}}\rangle\langle\frac{\vec{1}}{\sqrt{\left[n\atop k\right]}}, \vec{h_{0}}\rangle=\langle\vec{h_{0}}, \frac{\vec{1}}{\sqrt{\left[n\atop k\right]}}\rangle^{2}=\langle\vec{h}, \frac{\vec{1}}{\sqrt{\left[n\atop k\right]}}\rangle^{2}=\frac{(\vec{h}^{T}\vec{1})^{2}}{\left[n\atop k\right]}=\frac{|\mathcal{F}|^{2}}{\left[n\atop k\right]},
\end{array}
\end{equation*}
applying Lemmas \ref{c1}, \ref{c2} in the Appendix yields that
\begin{equation*}
\begin{array}{rl}
&~~~~~~\!0\overset{\eqref{e14}}{<}\left(b_{0}-\frac{a_{0}b_{1}}{a_{1}}\right)\frac{|\mathcal{F}|^{2}}{\left[n\atop k\right]}-\left(f-\frac{cb_{1}}{a_{1}}\right)|\mathcal{F}|+g\\[.3cm]
&\Leftrightarrow 0<\frac{q^{d^{2}}(q^{k-d}-1)(q^{n}-1)}{(q^{k}-1)(q^{n-d}-1)}\left[k\atop d\right]\left[n-d\atop k-d\right]\left[n-d\atop d\right]\frac{|\mathcal{F}|^{2}}{\left[n\atop k\right]}+q^{d^{2}}\left[n\atop d\right]\left[n-d\atop d\right]\left[n-d-1\atop k-d-1\right]^{2}\\[.3cm]
&~~~~~~~~~~~~~-\frac{q^{d^{2}}(2q^{k+n-d}-q^{n}-q^{n-d}-q^{k}-q^{k-d}+2)}{(q^{k}-1)(q^{n-d}-1)}\left[k\atop d\right]\left[n-d-1\atop k-d-1\right]\left[n-d\atop d\right]|\mathcal{F}|\\[.3cm]
&\Leftrightarrow 0<q^{d^{2}}\left[k\atop d\right]\left[n-d-1\atop k-d-1\right]\left[n-d\atop d\right]\frac{|\mathcal{F}|^{2}}{\left[n-1\atop k-1\right]}+q^{d^{2}}\left[n\atop d\right]\left[n-d\atop d\right]\left[n-d-1\atop k-d-1\right]^{2}\\[.3cm]
&~~~~~~~~~~~~~-\frac{q^{d^{2}}(2q^{k+n-d}-q^{n}-q^{n-d}-q^{k}-q^{k-d}+2)}{(q^{k}-1)(q^{n-d}-1)}\left[k\atop d\right]\left[n-d-1\atop k-d-1\right]\left[n-d\atop d\right]|\mathcal{F}|\\[.3cm]
&\Leftrightarrow 0<|\mathcal{F}|^{2}+\left[n-1\atop k-1\right]\cdot\frac{\left[n-d-1\atop k-d-1\right]\left[n\atop d\right]}{\left[k\atop d\right]}-\left[n-1\atop k-1\right]\cdot\frac{2q^{k+n-d}-q^{n}-q^{n-d}-q^{k}-q^{k-d}+2}{(q^{k}-1)(q^{n-d}-1)}|\mathcal{F}|\\[.3cm]
&\Leftrightarrow 0<|\mathcal{F}|^{2}+\left[n-1\atop k-1\right]\cdot\frac{\left[n-d-1\atop k-d-1\right]\left[n\atop d\right]}{\left[k\atop d\right]}-\left[n-1\atop k-1\right]\left(1+\frac{(q^{k-d}-1)(q^{n}-1)}{(q^{k}-1)(q^{n-d}-1)}\right)|\mathcal{F}|\\[.3cm]
&\Leftrightarrow 0<|\mathcal{F}|^{2}+\left[n-1\atop k-1\right]\cdot\frac{\left[n-d-1\atop k-d-1\right]\left[n\atop d\right]}{\left[k\atop d\right]}-\left(\left[n-1\atop k-1\right]+\left[n\atop k\right]\cdot\frac{q^{k-d}-1}{q^{n-d}-1}\right)|\mathcal{F}|\\[.3cm]
&\Leftrightarrow 0<|\mathcal{F}|^{2}+\left[n-1\atop k-1\right]\cdot\frac{\left[n-d-1\atop k-d-1\right]\left[n\atop d\right]}{\left[k\atop d\right]}-\left(\left[n-1\atop k-1\right]+\frac{\left[n-d-1\atop k-d-1\right]\left[n\atop d\right]}{\left[k\atop d\right]}\right)|\mathcal{F}|\\[.3cm]
&\Leftrightarrow 0<\left(|\mathcal{F}|-\left[n-1\atop k-1\right]\right)\left(|\mathcal{F}|-\frac{\left[n-d-1\atop k-d-1\right]\left[n\atop d\right]}{\left[k\atop d\right]}\right).
\end{array}
\end{equation*}
Since $\delta_{d}(\mathcal{F})> \left[n-d-1\atop k-d-1\right]$,
\begin{equation*}
|\mathcal{F}|\overset{\eqref{e3}}{=}\frac{\sum\limits_{S\in \left[V\atop d\right]}d_{S}(\mathcal{F})}{\left[k\atop d\right]}>\frac{\left[n\atop d\right]\left[n-d-1\atop k-d-1\right]}{\left[k\atop d\right]}.
\end{equation*}
Then we have $|\mathcal{F}|>\left[n-1\atop k-1\right]$, contradicting Theorem \ref{BBC}.
\qed
\section{ Concluding remarks }
In the present paper, we obtained a $d$-degree version of the Erd\H{o}s-Ko-Rado theorem for finite vector spaces. We proved that for $n\geq 2k+1$, the upper bound of the minimum $d$-degree of any intersecting family is $\left[n-d-1\atop k-d-1\right]$. It is very interesting to continue to study the structure of the intersecting family that reaches the upper bound.

%We say that two nonempty families $\mathcal{A}$ and $\mathcal{B}$ in $\mathcal{L}(V)$ are {\it cross-intersecting} if ${\rm dim}(S\cap T)\geq 1$ for all $S\in \mathcal{A}, T\in \mathcal{B}$.
 In \cite{deg}, Huang and Zhao combined a result from discrete geometry with eigenvalues of the Kneser graph to prove a degree version of Pyber's result \cite{Pyber} for cross-intersecting families of sets. It deserves to explore a degree version for cross-intersecting families in vector spaces using spectral graph theory.
\appendix
\setcounter{cla}{0}
\renewcommand{\thecla}{\arabic{cla}}
\section{Appendix}
%\section{Lemmas}
In this appendix we mainly prove some new Lemmas which are essential for our proofs. Firstly, we prove a lemma which compares three Gaussian binomial coefficients.
\begin{lem}\label{l200}
Let $k\geq i\geq d+1, n\geq 2k+1$. We have
\begin{equation*}
q^{\binom{i}{2}-ki}\left[n-k-i\atop k-i\right]\leq q^{\binom{d+1}{2}-k(d+1)}\left[n-k-d-1\atop k-d-1\right]< \frac{q^{k}-q^{d}}{q^{n}-q^{k}}\left[n-k\atop k\right]
\end{equation*}
\end{lem}

\begin{proof}~We first prove the first inequality. If $i=d+1$, it is obvious that the inequality holds. If $i>d+1$, we have
\begin{equation*}
\begin{array}{rl}
\frac{q^{\binom{d+1}{2}-k(d+1)}\left[n-k-d-1\atop k-d-1\right]}{q^{\binom{i}{2}-ki}\left[n-k-i\atop k-i\right]}\!\!\!\!&>q^{\binom{d+1}{2}-\binom{i}{2}+k(i-d-1)}\\[.3cm]
&=q^{(k-\frac{d+i}{2})(i-d-1)}\\[.3cm]
&>1
\end{array}
\end{equation*}
by $\left[n-k-d-1\atop k-d-1\right]=\prod\limits_{j=d+1}^{i-1}\frac{q^{n-k-j}-1}{q^{k-j}-1}\left[n-k-i\atop k-i\right]> \left[n-k-i\atop k-i\right]$, $k>d$ and $k\geq i$. Since $n\geq 2k+1$,
\begin{equation*}
\begin{array}{rl}
\frac{q^{k}-q^{d}}{q^{n}-q^{k}}\left[n-k\atop k\right]\!\!\!\!&=\frac{q^{k}-q^{d}}{q^{n}-q^{k}}\prod\limits_{j=0}^{d}\frac{q^{n-k-j}-1}{q^{k-j}-1}\left[n-k-d-1\atop k-d-1\right]\\[.3cm]
&= \frac{q^{k}-q^{d}}{q^{n}-q^{k}}\cdot \frac{q^{n-k}-1}{q^{k-d}-1}\prod\limits_{j=0}^{d-1}\frac{q^{n-k-1-j}-1}{q^{k-j}-1}\left[n-k-d-1\atop k-d-1\right]\\[.3cm]
&\geq q^{d-k}\left[n-k-d-1\atop k-d-1\right]\\[.3cm]
&>q^{d\left(\frac{d+1}{2}-k\right)-k}\left[n-k-d-1\atop k-d-1\right]\\[.3cm]
&=q^{\binom{d+1}{2}-k(d+1)}\left[n-k-d-1\atop k-d-1\right].\\[.3cm]
\end{array}
\end{equation*}
\end{proof}

\begin{lem}\label{c1}
Let $i=0, 1$, $a_{i}, b_{i}$ be defined in \eqref{a} and \eqref{b} respectively. Then
\begin{equation*}
b_{0}-\frac{a_{0}b_{1}}{a_{1}}=\frac{q^{d^{2}}(q^{k-d}-1)(q^{n}-1)}{(q^{k}-1)(q^{n-d}-1)}\left[k\atop d\right]\left[n-d\atop k-d\right]\left[n-d\atop d\right].
\end{equation*}
\end{lem}

\begin{proof}~By \eqref{a} and \eqref{b}, we have
\begin{equation*}
\begin{array}{rl}
a_{0}=\frac{q^{n}-q^{d}}{q^{n}-q^{k}}\left[n-k\atop k\right],~a_{1}=\frac{1-q^{d}}{q^{n}-q^{k}}\left[n-k\atop k\right],
\end{array}
\end{equation*}
\begin{equation*}
\begin{array}{rl}
b_{0}=q^{d^{2}}\left[k\atop d\right]\left[n-d\atop k-d\right]\left[n-d\atop d\right],~b_{1}=-q^{k+(d-1)^{2}-1}\left[k-1\atop d-1\right]\left[n-d-1\atop d-1\right]\left[n-d-1\atop k-d\right].
\end{array}
\end{equation*}
Then,
\begin{equation*}
\begin{array}{rl}
&~~~~b_{0}-\frac{a_{0}b_{1}}{a_{1}}\\[.3cm]
&=b_{0}\left(1-\frac{a_{0}}{a_{1}}\cdot\frac{b_{1}}{b_{0}}\right)\\[.3cm]
&=b_{0}\left(1-\frac{q^{n}-q^{d}}{1-q^{d}}\cdot\frac{-q^{k+(d-1)^{2}-1}\left[k-1\atop d-1\right]\left[n-d-1\atop d-1\right]\left[n-d-1\atop k-d\right]}{q^{d^{2}}\left[k\atop d\right]\left[n-d\atop d\right]\left[n-d\atop k-d\right]}\right)\\[.3cm]
&=b_{0}\left(1-\frac{q^{n}-q^{d}}{q^{d}-1}\cdot q^{k-2d}\cdot\frac{q^{d}-1}{q^{k}-1}\cdot\frac{q^{d}-1}{q^{n-d}-1}\cdot\frac{q^{n-k}-1}{q^{n-d}-1}\right)\\[.3cm]
&=b_{0}\left(1-q^{k-d}\cdot\frac{q^{d}-1}{q^{k}-1}\cdot\frac{q^{n-k}-1}{q^{n-d}-1}\right)\\[.3cm]
&=b_{0}\cdot\frac{(q^{k-d}-1)(q^{n}-1)}{(q^{k}-1)(q^{n-d}-1)},
\end{array}
\end{equation*}
the desired result follows.
\end{proof}

\begin{lem}\label{c2}
Let $a_{1}, b_{1}$, $f, c$ be defined in \eqref{a}, \eqref{b} and \eqref{cf} respectively. Then
\begin{equation*}
f-\frac{cb_{1}}{a_{1}}=\frac{q^{d^{2}}(2q^{k+n-d}-q^{n}-q^{n-d}-q^{k}-q^{k-d}+2)}{(q^{k}-1)(q^{n-d}-1)}\left[k\atop d\right]\left[n-d-1\atop k-d-1\right]\left[n-d\atop d\right].
\end{equation*}
\end{lem}

%\section*{Acknowledgements}

%The authors are grateful for the careful reading of two anonymous referees, who not only spotted some minor errors in the original version, but also gave helpful suggestions to improve the manuscript.
\begin{proof}~
We have
\begin{equation*}
\begin{array}{rl}
&~~~~f-\frac{cb_{1}}{a_{1}}\\[.3cm]
&=2q^{d^{2}}\left[k\atop d\right]\left[n-d\atop d\right]\left[n-d-1\atop k-d-1\right]-\frac{\frac{q^{k}-q^{d}}{q^{n}-q^{k}}\left[n-k\atop k\right]\cdot(-q^{k+(d-1)^{2}-1}\left[k-1\atop d-1\right]\left[n-d-1\atop d-1\right]\left[n-d-1\atop k-d\right])}{\frac{1-q^{d}}{q^{n}-q^{k}}\left[n-k\atop k\right]}\\[.3cm]
&=2q^{d^{2}}\left[k\atop d\right]\left[n-d\atop d\right]\left[n-d-1\atop k-d-1\right]-\frac{(q^{k}-q^{d})\cdot(q^{k+(d-1)^{2}-1}\left[k-1\atop d-1\right]\left[n-d-1\atop d-1\right]\left[n-d-1\atop k-d\right])}{q^{d}-1}\\[.3cm]
&=\left[k\atop d\right]\left[n-d\atop d\right]\left[n-d-1\atop k-d-1\right]\left(2q^{d^{2}}-\frac{q^{k+(d-1)^{2}-1}(q^{k}-q^{d})}{q^{d}-1}\cdot\frac{q^{d}-1}{q^{k}-1}\cdot\frac{q^{d}-1}{q^{n-d}-1}\cdot\frac{q^{n-k}-1}{q^{k-d}-1}\right)\\[.3cm]
&=q^{d^{2}}\left[k\atop d\right]\left[n-d\atop d\right]\left[n-d-1\atop k-d-1\right]\left(2-\frac{q^{k-d}(q^{d}-1)(q^{n-k}-1)}{(q^{k}-1)(q^{n-d}-1)}\right)\\[.3cm]
&=\frac{q^{d^{2}}(2q^{k+n-d}-q^{n}-q^{n-d}-q^{k}-q^{k-d}+2)}{(q^{k}-1)(q^{n-d}-1)}\left[k\atop d\right]\left[n-d-1\atop k-d-1\right]\left[n-d\atop d\right].
\end{array}
\end{equation*}
by \eqref{a}, \eqref{b} and \eqref{cf}.
\end{proof}

Let $$[x]_{i}:=\prod\limits_{j=0}^{i-1}(q^{x-j}-1),$$ $$S_{i}(n)=\frac{q^{\binom{i}{2}-ki+k}[k-1]_{i-1}}{[n-k-1]_{i-1}},$$
$$T_{i}(n)=\frac{q^{(k-2d)(i-1)+\binom{i}{2}}[d-1]_{i-1}^{2}[n-k-1]_{i-1}}{[k-1]_{i-1}[n-d-1]_{i-1}^{2}}.$$
We establish the following Lemmas.
\begin{lem}\label{c3}
Let $3\leq i\leq d<k$ and $2k \leq n$. Then
\begin{equation*}
(q^{k}-1)S_{i}(n)-(q^{d}-1)T_{i}(n)<q^{k}-q^{d}
\end{equation*}
\end{lem}

\begin{proof}~
For fixed $k$, $d$ and $n$, let
\begin{equation*}
\alpha_{i}=\frac{q^{\binom{i}{2}-ki+k}[k-1]_{i-1}}{[n-k-1]_{i-1}},
\end{equation*}

\begin{equation*}
\beta_{i}=\frac{q^{\binom{i}{2}-di+d}[d-1]_{i-1}}{[n-d-1]_{i-1}}.
\end{equation*}
Then $S_{i}(n)=\alpha_{i}$, $T_{i}(n)=\frac{\beta_{i}^{2}}{\alpha_{i}}$.
Note that for $3\leq i\leq d<k$,
\begin{equation}\label{e16}
\frac{\alpha_{i+1}}{\alpha_{i}}=\frac{1-q^{i-k}}{q^{n-k-i}-1}<1.
\end{equation}
Since $\frac{\beta_{i+1}}{\beta_{i}}\cdot \frac{\alpha_{i}}{\alpha_{i+1}}=\frac{(1-q^{i-d})(q^{n-k-i}-1)}{(q^{n-d-i}-1)(1-q^{i-k})}<\frac{q^{n-k-i}-1}{q^{n-d-i}-1}<1$,
\begin{equation}\label{e19}
\frac{\beta_{i+1}}{\alpha_{i+1}}<\frac{\beta_{i}}{\alpha_{i}}.
\end{equation}
Since $[k-1]_{i-1}>q^{(k-d)(i-1)}[d-1]_{i-1}$, $[n-d-1]_{i}>q^{(k-d)i}[n-k-1]_{i}$ and $3\leq i\leq d<k$,
\begin{equation*}
\begin{array}{rl}
\alpha_{i}-\alpha_{i+1}\!\!\!\!&=\left(1-\frac{\alpha_{i+1}}{\alpha_{i}}\right)\alpha_{i}\\[.3cm]
&\!\!\overset{\eqref{e16}}{=}\left(1-\frac{1-q^{i-k}}{q^{n-k-i}-1}\right)\alpha_{i}\\[.3cm]
&=\frac{q^{n-k-i}+q^{i-k}-2}{q^{n-k-i}-1}\alpha_{i}\\[.3cm]
&=\frac{q^{n-k-i}+q^{i-k}-2}{q^{n-k-i}-1}\cdot\frac{q^{\binom{i}{2}-ki+k}[k-1]_{i-1}}{[n-k-1]_{i-1}}\\[.3cm]
&=\frac{q^{\binom{i}{2}-ki+k}[k-1]_{i-1}(q^{n-k-i}+q^{i-k}-2)}{[n-k-1]_{i}}\\[.3cm]
&>\frac{q^{\binom{i}{2}-di+d}[d-1]_{i-1}(q^{n-k-i}+q^{i-k}-2)}{[n-k-1]_{i}}\\[.3cm]
&>\frac{q^{\binom{i}{2}-di+d}[d-1]_{i-1}(q^{n-k-i}+q^{i-k}-2)\cdot q^{(k-d)i}}{[n-d-1]_{i}}\\[.3cm]
&>\frac{q^{\binom{i}{2}-di+d}[d-1]_{i-1}(q^{n-d-i}+q^{i-d}-2)}{[n-d-1]_{i}}\\[.3cm]
&=\left(1-\frac{1-q^{i-d}}{q^{n-d-i}-1}\right)\beta_{i}\\[.3cm]
&=\left(1-\frac{\beta_{i+1}}{\beta_{i}}\right)\beta_{i}\\[.3cm]
&=\beta_{i}-\beta_{i+1}.
\end{array}
\end{equation*}
Then
\begin{equation}\label{e20}
\alpha_{i}-\beta_{i}>\alpha_{i+1}-\beta_{i+1}.
\end{equation}
Since
\begin{equation*}
\begin{array}{rl}
(q^{k}-1)S_{i}(n)-(q^{d}-1)T_{i}(n)\!\!\!\!&=(q^{k}-q^{d})S_{i}(n)+(q^{d}-1)\left(S_{i}(n)-T_{i}(n)\right)\\[.3cm]
&=(q^{k}-q^{d})\alpha_{i}+(q^{d}-1)\left(\alpha_{i}-\frac{\beta_{i}^{2}}{\alpha_{i}}\right)\\[.3cm]
&=(q^{k}-q^{d})\alpha_{i}+(q^{d}-1)(\alpha_{i}-\beta_{i})(1+\frac{\beta_{i}}{\alpha_{i}})\\[.3cm]
\end{array}
\end{equation*}
and note \eqref{e16}, \eqref{e19} and \eqref{e20}, it suffices to check the case $i=3$. %We have
%\begin{equation*}
%\begin{array}{rl}
%&~~~~(q^{k}-1)S_{i}(n)-(q^{d}-1)T_{i}(n)<q^{k}-q^{d}\\[.3cm]
%&\Leftrightarrow (q^{k}-1)\alpha_{3}-(q^{d}-1)\frac{\beta_{3}^{2}}{\alpha_{3}}<q^{k}-q^{d}\\[.3cm]
%&\Leftrightarrow (q^{k}-1)\alpha_{3}-(q^{k}-q^{d})\alpha_{3}-(q^{d}-1)\frac{\beta_{3}^{2}}{\alpha_{3}}<(q^{k}-q^{d})(1-\alpha_{3})\\[.3cm]
%&\Leftrightarrow (q^{d}-1)\alpha_{3}-(q^{d}-1)\frac{\beta_{3}^{2}}{\alpha_{3}}<(q^{k}-q^{d})(1-\alpha_{3})\\[.3cm]
%&\Leftrightarrow (q^{d}-1)\left(\alpha_{3}-\frac{\beta_{3}^{2}}{\alpha_{3}}\right)<(q^{k}-q^{d})(1-\alpha_{3})\\[.3cm]
%&\Leftrightarrow (q^{d}-1)\frac{\alpha_{3}^{2}-\beta_{3}^{2}}{\alpha_{3}}<(q^{k}-q^{d})(1-\alpha_{3})\\[.3cm]
%&\Leftrightarrow (q^{d}-1)(\alpha_{3}-\beta_{3})\frac{\alpha_{3}+\beta_{3}}{\alpha_{3}}<(q^{k}-q^{d})(1-\alpha_{3})\\[.3cm]
%&\Leftrightarrow (q^{d}-1)(\alpha_{3}-\beta_{3})\left(1+\frac{\beta_{3}}{\alpha_{3}}\right)<(q^{k}-q^{d})(1-\alpha_{3})\\[.3cm]
%\end{array}
%\end{equation*}
Since
%\begin{equation*}
%\begin{array}{rl}
%1-\alpha_{3}=1-\frac{q^{3-2k}[k-1]_{2}}{[n-k-1]_{2}}=1-\frac{q^{3-2k}(q^{k-1}-1)(q^{k-2}-1)}{(q^{n-k-1}-1)(q^{n-k-2}-1)}
%\end{array}
%\end{equation*}
\begin{equation*}
\begin{array}{rl}
(q^{k}-1)S_{3}(n)-(q^{d}-1)T_{3}(n)\!\!\!\!&<(q^{k}-1)S_{3}(n)\\[.3cm]
&=(q^{k}-1)\cdot\frac{q^{3-2k}[k-1]_{2}}{[n-k-1]_{2}}\\[.3cm]
&=\frac{q^{3-2k}(q^{k}-1)(q^{k-1}-1)(q^{k-2}-1)}{(q^{n-k-1}-1)(q^{n-k-2}-1)}\\[.3cm]
&\leq q^{3-2k}(q^{k}-1)\\[.3cm]
&< q^{k}-q^{d},
\end{array}
\end{equation*}
the result holds.
\end{proof}

\begin{lem}\label{c4}
Let $2\leq i\leq d<k$, $2k\leq n$ and $a_{i}, b_{i}$ be defined in \eqref{a} and \eqref{b} respectively. Then
\begin{equation*}
b_{i}<\frac{a_{i}b_{1}}{a_{1}}.
\end{equation*}
\end{lem}

\begin{proof}~Since $b_{1}=-q^{k+(d-1)^{2}-1}\left[k-1\atop d-1\right]\left[n-d-1\atop d-1\right]\left[n-d-1\atop k-d\right]<0$ and $a_{1}=\frac{1-q^{d}}{q^{n}-q^{k}}\left[n-k\atop k\right]$,
\begin{equation*}
\begin{array}{rl}
&~~~~b_{i}<\frac{a_{i}b_{1}}{a_{1}}\\[.3cm]
&\Leftrightarrow\frac{b_{i}}{-b_{1}}<\frac{a_{i}}{-a_{1}}\\[.3cm]
&\Leftrightarrow\frac{(-1)^{i}q^{ki+(d-i)^{2}-\binom{i+1}{2}}\left[k-i\atop d-i\right]\left[n-d-i\atop d-i\right]\left[n-d-i\atop k-d\right]}{q^{k+(d-1)^{2}-1}\left[k-1\atop d-1\right]\left[n-d-1\atop d-1\right]\left[n-d-1\atop k-d\right]}<\frac{\frac{q^{k}-q^{d}}{q^{n}-q^{k}}\left[n-k\atop k\right]+(-1)^{i}q^{\binom{i}{2}-ki}\left[n-k-i\atop k-i\right]}{\frac{q^{d}-1}{q^{n}-q^{k}}\left[n-k\atop k\right]}\\[.3cm]
&\Leftrightarrow\frac{(-1)^{i}q^{(k-2d)(i-1)+\binom{i}{2}}\left[k-i\atop d-i\right]\left[n-d-i\atop d-i\right]\left[n-d-i\atop k-d\right]}{\left[k-1\atop d-1\right]\left[n-d-1\atop d-1\right]\left[n-d-1\atop k-d\right]}<\frac{(q^{k}-q^{d})+(-1)^{i}q^{\binom{i}{2}-ki}(q^{n}-q^{k})\frac{\left[n-k-i\atop k-i\right]}{\left[n-k\atop k\right]}}{q^{d}-1}\\[.3cm]
&\Leftrightarrow\frac{(-1)^{i}q^{(k-2d)(i-1)+\binom{i}{2}}[d-1]_{i-1}^{2}[n-k-1]_{i-1}}{[k-1]_{i-1}[n-d-1]_{i-1}^{2}}<\frac{(q^{k}-q^{d})+(-1)^{i}q^{\binom{i}{2}-ki+k}(q^{k}-1)\frac{[k-1]_{i-1}}{[n-k-1]_{i-1}}}{q^{d}-1}\\[.3cm]
&\Leftrightarrow (-1)^{i}T_{i}(n)<\frac{(q^{k}-q^{d})+(-1)^{i}(q^{k}-1)S_{i}(n)}{q^{d}-1}\\[.3cm]
&\Leftrightarrow (-1)^{i}\left((q^{d}-1)T_{i}(n)-(q^{k}-1)S_{i}(n)\right)<q^{k}-q^{d}.
\end{array}
\end{equation*}

We can distinguish the cases when $i$ is odd and even. Let us consider the first case when $i$ is even.
Since $d<k$, $2k\leq n$ and $\frac{q^{s}-1}{q^{t}-1}> q^{s-t}$ for $s> t$.
\begin{equation*}
\begin{array}{rl}
\frac{S_{i}(n)}{T_{i}(n)}\!\!\!\!&=\frac{[k-1]_{i-1}^{2}[n-d-1]_{i-1}^{2}}{q^{(2k-2d)(i-1)}[d-1]_{i-1}^{2}[n-k-1]_{i-1}^{2}}\\[.3cm]
&>\frac{[n-d-1]_{i-1}^{2}}{[n-k-1]_{i-1}^{2}}\\[.3cm]
&>1,
\end{array}
\end{equation*}
that is, $S_{i}(n)>T_{i}(n)$. Hence,
$$(q^{d}-1)T_{i}(n)-(q^{k}-1)S_{i}(n)<(q^{d}-1)S_{i}(n)-(q^{k}-1)S_{i}(n)=(q^{d}-q^{k})S_{i}(n)<0<q^{k}-q^{d},$$
which implies the desired result.

Now suppose $i$ is odd and thus $i\geq3$. Then
\begin{equation*}
b_{i}<\frac{a_{i}b_{1}}{a_{1}}\Leftrightarrow (q^{k}-1)S_{i}(n)-(q^{d}-1)T_{i}(n)<q^{k}-q^{d}.
\end{equation*}
Hence the result holds by Lemma \ref{c3}.
\end{proof}

\end{document}